\documentclass[11pt, oneside]{amsart}
 \usepackage[text={5.58in,8.5in},centering,letterpaper,dvips]{geometry}
\usepackage{graphicx}
\usepackage{amsfonts}
\usepackage{epsf}
\usepackage{amssymb}
\usepackage{amsmath}
\usepackage{amscd}
\usepackage{tikz}
\usepackage{pdfpages}
\usepackage{fancyhdr}
\usepackage{setspace}
\usepackage{hyperref}
\usepackage[all]{xy}
\usetikzlibrary{matrix}
\usepackage{verbatim}
\usepackage{enumerate}
\usepackage{amsbsy}

\theoremstyle{theorem}
\newtheorem{theorem}{Theorem}[section]
\newtheorem{proposition}[theorem]{Proposition}
\newtheorem{lemma}[theorem]{Lemma}

\newtheorem{corollary}[theorem]{Corollary}
\newtheorem{conjecture}[theorem]{Conjecture}

\newtheorem{problems}[theorem]{Problems}

\theoremstyle{definition}
\newtheorem{definition}[theorem]{Definition}

\newcommand{\Z}{\mathbb{Z}}

\newcommand{\Pp}{\mathbb{P}}

\newcommand{\R}{\mathbb{R}}
\newcommand{\RP}{\mathbb{RP}}
\newcommand{\CP}{\mathbb{CP}}

\newcommand{\Cc}{\mathcal C}
\newcommand{\Dd}{\mathcal D}
\newcommand{\Tt}{\mathcal T}

\newcommand{\Hh}{\mathcal H}
\newcommand{\Kk}{\mathcal K}

\newcommand{\Qq}{\mathcal Q}

\newcommand{\A}{\alpha}
\newcommand{\n}{\beta}
\newcommand{\g}{\gamma}
\newcommand{\pd}{\partial}

\newcommand{\emp}{\emptyset}
\newcommand{\X}{\times}

\newcommand{\Bb}{\mathcal{B}}
\newcommand{\Sss}{\mathcal{S}}
\newcommand{\eps}{\epsilon}

\makeatletter
\def\@seccntformat#1{%
  \protect\textup{\protect\@secnumfont
    \ifnum\pdfstrcmp{subsection}{#1}=0 \bfseries\fi
    \csname the#1\endcsname
    \protect\@secnumpunct
  }%
}  
\makeatother

\makeatletter
\newtheorem*{rep@theorem}{\rep@title}
\newcommand{\newreptheorem}[2]{%
\newenvironment{rep#1}[1]{%
 \def\rep@title{#2 \ref{##1}}%
 \begin{rep@theorem}}%
 {\end{rep@theorem}}}
\makeatother

\newreptheorem{theorem}{Theorem}
\newreptheorem{lemma}{Lemma}
\newreptheorem{question}{Question}
\newreptheorem{corollary}{Corollary}
\newreptheorem{proposition}{Proposition}
\newreptheorem{problem}{Problem}

\makeatletter
\newcommand*\wt[1]{\mathpalette\wthelper{#1}}
\newcommand*\wthelper[2]{%
        \hbox{\dimen@\accentfontxheight#1%
                \accentfontxheight#11.3\dimen@
                $\m@th#1\widetilde{#2}$%
                \accentfontxheight#1\dimen@
        }%
}

\newcommand*\accentfontxheight[1]{%
        \fontdimen5\ifx#1\displaystyle
                \textfont
        \else\ifx#1\textstyle
                \textfont
        \else\ifx#1\scriptstyle
                \scriptfont
        \else
                \scriptscriptfont
        \fi\fi\fi3
}
\makeatother

\begin{document}

\rhead{\thepage}
\lhead{\author}
\thispagestyle{empty}


\raggedbottom
\pagenumbering{arabic}
\setcounter{section}{0}


\title{Bridge trisections of knotted surfaces in 4--manifolds}
\date{\today}

\author{Jeffrey Meier}
\address{Department of Mathematics, University of Georgia, Athens, GA 30606}
\email{jeffrey.meier@uga.edu}
\urladdr{jeffreymeier.org} 

\author{Alexander Zupan}
\address{Department of Mathematics, University of Nebraska-Lincoln, Lincoln, NE 68588}
\email{zupan@unl.edu}
\urladdr{http://www.math.unl.edu/~azupan2}

\begin{abstract}
	We prove that every smoothly embedded surface in a 4--manifold can be isotoped to be in bridge position with respect to a given trisection of the ambient 4--manifold; that is, after isotopy, the surface meets components of the trisection in trivial disks or arcs.  Such a decomposition, which we call a \emph{generalized bridge trisection}, extends the authors' definition of bridge trisections for surfaces in $S^4$.  Using this new construction, we give diagrammatic representations called \emph{shadow diagrams} for knotted surfaces in 4--manifolds.  We also provide a low-complexity classification for these structures and describe several examples, including the important case of complex curves inside $\mathbb{CP}^2$.  Using these examples, we prove that there exist exotic 4--manifolds with $(g,0)$--trisections for certain values of $g$.  We conclude by sketching a conjectural uniqueness result that would provide a complete diagrammatic calculus for studying knotted surfaces through their shadow diagrams.
\end{abstract}

\maketitle

\section{Introduction}


Every knot in $S^3$ can be cut into two trivial tangles (collections of unknotted arcs) in a classical decomposition known as a \emph{bridge splitting}.  This structure provides a convenient measure of complexity, the number of unknotted arcs in each collection, and the smallest number of such arcs in any bridge splitting of a given knot $K$ is the widely studied \emph{bridge number} of $K$.  It is well-known that the idea of a bridge splitting can be extended to other spaces:  Every 3--manifold $Y$ admits a Heegaard splitting, a decomposition of $Y$ into two simple pieces called handlebodies, and given a knot $K \subset Y$, there is an isotopy of $K$ after which it meets each handlebody in a collection of unknotted arcs.

In dimension four, decompositions analogous to Heegaard splittings cut spaces into not two but three components.  Gay and Kirby proved that every smooth, closed, connected, orientable 4--manifold (henceforth, \emph{4--manifold}) $X$ admits a trisection, splitting $X$ into three simple 4--dimensional pieces (4--dimensional 1--handlebodies) that meet pairwise in 3--dimensional handlebodies and have as their common intersection a closed surface.  Similarly, in~\cite{Meier-Zupan_Bridge_2017} the authors proved that every smoothly embedded, closed surface (henceforth, \emph{knotted surface}) $\Kk$ in $S^4$ admits a bridge trisection, a decomposition of the pair $(S^4,\Kk)$ into three collections of unknotted disks in 4--balls that intersect in trivial tangles in 3--balls, akin to classical bridge splittings in $S^3$.  In this paper, we extend this construction to knotted surfaces in arbitrary 4--manifolds.  Given a trisection $\Tt$ splitting a 4--manifold $X$ into $X_1 \cup X_2 \cup X_3$, we say that a knotted surface $\Kk \subset X$ is in \emph{bridge position} if $\Kk \cap X_i$ is a collection of unknotted disks and $\Kk \cap (X_i \cap X_j)$ is a collection of trivial tangles.  Our first result is the following.

\begin{theorem}\label{mainthm}
	Let $X$ be a 4--manifold with trisection $\Tt$.  Any knotted surface $\Kk$ in $X$ can be isotoped into bridge position with respect to $\Tt$.
\end{theorem}

If $\Kk \subset X$ is in bridge position with respect to a trisection, we call the decomposition $(X,\Kk) = (X_1,\Kk \cap X_1) \cup (X_2,\Kk \cap X_2) \cup (X_3,\Kk \cap X_3)$ a \emph{generalized bridge trisection}. 

Returning to dimension three, we note that it can be fruitful to modify a bridge splitting of a knot $K$ in a 3--manifold $Y$ so that the complexity of the underlying Heegaard splitting increases while the number of unknotted arcs decreases.  This process involves a technical operation called \emph{meridional stabilization}.  We show that there is an analogous operation, which we also call meridional stabilization, in the context of bridge trisections.  As a result, we prove the next theorem. (Precise definitions are included in Section~\ref{sec:prelim}.)

\begin{theorem}\label{mainthm2}
	Let $\Kk$ be a knotted surface with $n$ connected components in a 4--manifold $X$. The pair $(X,\Kk)$ admits a $(g,k;b,n)$--generalized bridge trisection satisfying $b=3n-\chi(\Kk)$.  In particular, if $\Kk$ is a 2--knot in $X$, then $\Kk$ can be put in 1--bridge position.
\end{theorem}

A generalized bridge trisection of the type guaranteed by Theorem~\ref{mainthm2} is called \emph{efficient} with respect to the underlying trisection $\Tt$, since it is the smallest possible $b$ and $n$ for any surface with the same Euler characteristic.

As a Corollary to Theorem~\ref{mainthm}, we explain how these new decompositions provide a way to encode a knotted surface combinatorially in a 2--dimensional diagram, which we call a \emph{shadow diagram}.  We anticipate that this new paradigm in the study of knotted surfaces will open a window to novel structures and connections in this field.

\begin{corollary}\label{maincor}
	Every generalized bridge trisection of a knotted surface $\Kk$ in a 4--manifold $X$ induces a shadow diagram.  Moreover, if $\Kk$ has $n$ components, then an efficient generalized bridge trisection of $\Kk$ induces a shadow diagram with $9n-3\chi(\Kk)$ arcs.  In particular, if $\Kk$ is a 2--knot in $X$, then $(X,\Kk)$ admits a doubly-pointed trisection diagram.
\end{corollary}

A knot that has been decomposed into a pair unknotted arcs admits a representation called a \emph{doubly pointed Heegaard diagram}; a doubly-pointed trisection diagram is a direct adaptation of this structure.  See Section~\ref{sec:exs} for further details.

In Section~\ref{sec:exs}, we give shadow diagrams for various examples of simple surfaces in 4--manifolds.  First, we give a classification of those 2--knots that can be put in 1--bridge position with respect to a genus one trisection of the ambient 4--manifold; the only non-trivial examples are the complex line $\CP^1$ and the quadric $\Cc_2$, both in $\CP^2$.  We then expand this to a the study of complex curves in complex 4--manifolds, announcing preliminary results related to on-going work with Peter Lambert-Cole.  In particular, we announce the following result, which shows that complex curves in $\CP^2$ have efficient generalized bridge trisections with respect to the genus one trisection of $\CP^2$.


\begin{theorem}\label{thm:complex}
Let $\Cc_d$ be the complex curve of degree $d$ in $\CP^2$.  Then, the pair $(\CP^2,\Cc_d)$ admits an efficient generalized bridge trisection of genus one.
\end{theorem}

This theorem can be used to prove the existence of \emph{efficient exotic trisections}, which in this setting are defined to be $(g,0)$-trisections of exotic 4--manifolds homeomorphic but not diffeomorphic to some standard 4--manifold.  For example, we show that the degree $d$ complex surface in $\CP^3$, which is homeomorphic, but not diffeomorphic, to a connected sum of $\CP^2$s and $\overline{\CP^2}$s for odd $d\geq 5$~\cite{Don_1990_Polynomial-invariants}, admits a $(g,0)$--trisection, where $g=d(d-2)(d-1)^2+d$.  The first examples of (inefficient) exotic trisections were given by Baykur and Saeki~\cite{Baykur-Saeki_Simplifying_2017}. 

Section~\ref{sec:proofs} contains the proofs of the main theorems and corollaries.  In Section~\ref{sec:unique}, we turn our attention to the question of uniqueness of generalized bridge trisections.  To this end, we offer the following conjecture, an affirmative answer to which seems within reach, if one can trace through the relevant Cerf theory.

\begin{conjecture}\label{conj:perturb}
	Any two generalized bridge trisections for a pair $(X,\Kk)$ that induce isotopic trisections of $X$ can be made isotopic after a sequence of elementary perturbation and unperturbation moves.
\end{conjecture}

\subsection*{Acknowledgements}

The authors would like to thank Rob Kirby for posing the question the inspired this paper, Peter Lambert-Cole for his interest in this work and for graciously sharing his beautiful shadow diagrams for complex curves in $\CP^2$, and John Baldwin for inquiring about a trisection diagram for K3, which sparked a sequence of realizations that led to Theorem~\ref{thm:complex} and its corollaries.  
The first author is supported by NSF grants DMS-1400543 and DMS-1758087, and the second author is supported by NSF grant DMS-1664578 and NSF-EPSCoR grant OIA-1557417.

\section{Preliminaries}\label{sec:prelim}

We will work in the smooth category throughout this paper.  All 4--manifolds are assumed to be orientable.  Let $\nu( \cdot )$ denote an open regular neighborhood in an ambient manifold that should be clear from context.  A \emph{knotted surface} $\Kk$ in a 4--manifold $X$ is a smoothly embedded, closed surface, possibly disconnected and possibly non-orientable, considered up to smooth isotopy in $X$. We will often refer to handlebodies in dimensions three and four; except where a further distinction is appropriate, we will use the term \emph{handlebody} to refer to $\natural^g(S^1\times D^2)$ and the term \emph{1--handlebody} to refer to $\natural^k(S^1\times B^3)$; by the \emph{genus} of these objects, we mean $g$ and $k$, respectively.

A \emph{trisection} $\Tt$ of a closed 4--manifold $X$, introduced by Gay and Kirby~\cite{Gay-Kirby_Trisecting_2016}, is a decomposition $X = X_1 \cup X_2 \cup X_3$, where $X_i$ is a 1--handlebody, $H_{ij} = X_i \cap X_j$ is a handlebody for $i \neq j$, and $\Sigma = X_1 \cap X_2 \cap X_3$ is a closed surface.  A trisection is uniquely determined by its \emph{spine}, $H_{12} \cup H_{23} \cup H_{31}$, and the spine of a trisection can be encoded with a \emph{trisection diagram} $(\A,\n,\g)$, a collection of three cut systems $\A,\n,\g$ on the surface $\Sigma$ yielding the three handlebodies $H_{31}$, $H_{12}$, $H_{23}$, respectively. (A \emph{cut system} in a genus $g$ surface $\Sigma$ is a collection of $g$ pairwise disjoint curves cutting $\Sigma$ into a planar surface, and attaching 2--handles to $\Sigma$ along a cut system yields a handlebody.)  Sometimes it will be useful to assign a complexity to a trisection $\Tt$: If $g$ is the genus of the central surface $\Sigma$ and $k_i$ is the genus of the 1--handlebody $X_i$, we call $\Tt$ a $(g;k_1,k_2,k_3)$--trisection.  In the case that $k_1 = k_2 = k_3$, we call $\Tt$ a $(g,k)$--trisection (with $k = k_1$).

A collection of properly embedded arcs $\tau = \{\tau_i\}$ in the handlebody $H$ is \emph{trivial} if there is an isotopy carrying $\tau$ into $\pd H$.  Equivalently, there is a collection of pairwise disjoint disks $\Delta=\{\Delta_i\}$, called \emph{bridge disks}, such that $\pd \Delta_i$ is the endpoint union of $\tau_i$ and an arc $\tau_i'$ in $\pd H$.  The arc $\tau_i'$ is called a \emph{shadow} of $\tau_i$.  We also call a collection of trivial arcs a \emph{trivial tangle}.  Let $L$ be a link in a 3--manifold $Y$.  A \emph{bridge splitting} of $(Y,L)$ is a decomposition $(Y,L) = (H_1,\tau_1) \cup_{\Sigma} (H_2,\tau_2)$, where $H_i$ is a handlebody containing a trivial tangle $\tau_i$ and $\Sigma = H_1 \cap H_2$.  It is well-known that every pair $(Y,L)$ admits a bridge splitting.

Moving to dimension four, a collection $\Dd$ of properly embedded disks in a 1--handlebody $V$ is \emph{trivial} if the disks $\Dd$ are simultaneously isotopic into $\pd V$.  Let $\Kk$ be a knotted surface in a closed 4--manifold $X$.

\begin{definition}
	A \emph{generalized bridge trisection} of the pair $(X,\Kk)$ is a decomposition $(X,\Kk) = (X_1,\Dd_1) \cup (X_2, \Dd_2) \cup (X_3,\Dd_3)$, where $X = X_1 \cup X_2 \cup X_3$ is a trisection, $\Dd_i$ is a collection of trivial disks in $X_i$, and for $i\not=j$, the arcs $\tau_{ij} = \Dd_i \cap \Dd_j$ form a trivial tangle in $H_{ij}$.
\end{definition}

  In~\cite{Meier-Zupan_Bridge_2017}, the authors proved that every knotted surface in $S^4$ admits a generalized bridge trisection in which the underlying trisection of $S^4$ is the standard genus zero trisection.  We will refer to such a decomposition simply as a \emph{bridge trisection}.  The present article extends this theorem to a given trisection of an arbitrary 4--manifold.  
  
\begin{definition}
	If $\Tt$ is a trisection of $X$ given by $X = X_1 \cup X_2 \cup X_3$, and $\Kk$ is a knotted surface in $X$ such that $(X,\Kk) = (X_1, \Kk \cap X_1) \cup (X_2, \Kk \cap X_2) \cup (X_3, \Kk \cap X_3)$ is a generalized bridge trisection, we say that $\Kk$ is in \emph{bridge position} with respect to $\Tt$.
\end{definition}  

The union $(H_{12},\tau_{12}) \cup (H_{23},\tau_{23}) \cup (H_{31},\tau_{31})$ is called the \emph{spine} of a generalized bridge trisection.  As is the case with trisections, bridge trisections are uniquely determined by their spines.  Fortunately, the same is true for generalized bridge trisections.  To prove this fact, we need the following lemma.

\begin{lemma}\label{unq}
Let $V$ be a 1--handlebody, and let $L$ be an unlink contained in $\pd V$.  Up to an isotopy fixing $L$, the unlink $L$ bounds a unique collection trivial disks in $V$.
\end{lemma}
\begin{proof}
It is well-known that the statement is true when $V$ is a 4--ball~\cite{Livingston_Surfaces_1982}.  Suppose that $\Dd$ and $\Dd'$ are two collections of trivial disks properly embedded in $V$ such that $\pd \Dd = \pd \Dd' = L$.  Let $\Bb$ denote a collection of properly embedded 3--balls in $V$ cutting $V$ into a 4--ball, and let $\Sss = \pd \Bb$.  Since $L$ is an unlink, we may isotope $\Sss$ in $\pd V$ (along with $\Bb$ in $V$) so that $L \cap \Sss = \emp$.  Since $\Dd$ is a collection of boundary parallel disks in $V$, there exists a set of disks $\Dd_* \subset \pd V$ isotopic to $\Dd$ via an isotopy fixing $L$.  Choose $\Dd_*$ so that the number of components of $\Dd_* \cap \Sss$ is minimal among all such sets of disks.

We claim that $\Dd_* \cap \Sss = \emp$.  First, we observe that every embedded 2--sphere $S \subset \pd V$ bounds a properly embedded 3--ball in $V$:  If $S$ does not bound a 3--ball in $\pd V$, then either $S$ is an essential separating sphere, splitting $\pd V$ into two components, each of which is a connected sum of copies of $S^1 \X S^2$, or $S$ is an essential nonseparating sphere, and there is an $S^1 \X S^2$ summand of $\pd V$ in which $S$ is isotopic to $\{\text{pt}\} \X S^2$.  In either case, we can cap off $\pd V$ with 4--dimensional 3--handles and a 4--handle to obtain a 1--handlebody in which $S$ bounds a 3--ball.  However, this capping off process is unique~\cite{Laudenbach-Poenaru_A-note_1972}, and thus $S$ bounds a 3--ball in $V$ as well.

To prove the claim, suppose by way of contradiction that $\Dd_* \cap \Sss \neq \emp$, and choose a curve $c$ of $\Dd_* \cap \Sss$ that is innermost in a sphere component of $\Sss$, so $c$ bounds a disk $E$ in this component such that $\text{int}(E) \cap \Dd_* = \emp$.  Note that $c$ also bounds a sub-disk $D$ of a component of $\Dd_*$, so $S = E \cup D$ is a 2--sphere embedded in $\pd V$.  By the above argument, $S$ bounds a 3--ball $B$ that is properly embedded in $V$; thus, $\text{int}(B) \cap \Dd_* = \emp$.  It follows that there is an isotopy of $\Dd_*$ through $B$ in $V$ that pushes $D$ onto $E$.  If $\Dd'_*$ is the set of disks obtained from $\Dd_*$ by removing $D$ and gluing on a copy of $E$ (pushed slightly off of $\Sss$), then $\Dd_*'$ is isotopic to $\Dd_*$, with $|\Dd'_* \cap \Sss| < |\Dd_* \cap \Sss|$, a contradiction.

It follows that $\Dd_* \cap \Sss = \emp$, and we conclude that after isotopy $\Dd$ is contained in the 4--ball $W = V\setminus\nu(\Bb)$.  Similarly, we can assume that after isotopy $\Dd'$ is contained in $W$.  It now follows from~\cite{Livingston_Surfaces_1982} that $\Dd$ and $\Dd'$ can are isotopic, as desired.
\end{proof}

\begin{corollary}
A generalized bridge trisection is uniquely determined by its spine.
\end{corollary}
\begin{proof}
The set $H_{12} \cup H_{23} \cup H_{31}$ determines the trisection of the underlying 4--manifold $X$, and pairs of the trivial tangles $\tau_{12}$, $\tau_{23}$, and $\tau_{31}$ determine the unlinks $\pd \Dd_i \subset \pd X_i$.  By Lemma~\ref{unq}, the disks $\Dd_i$ are determined up to isotopy by $\pd \Dd_i$, and thus the spine of the bridge trisection yields the remainder of its components.
\end{proof}

We also observe that we can compute the Euler characteristic of a surface $\Kk$ from the parameters of a generalized bridge trisection.  If $\Kk \subset X$ is in bridge position with respect to a trisection $\Tt$ of $X$, we will set the convention that $c_i = |\Kk \cap X_i|$ and $b = |\Kk \cap H_{ij}| = |\Kk \cap \Sigma|/2$.

\begin{lemma}\label{echar}
Suppose that $\Kk \subset X$ is in bridge position with respect to a trisection $\Tt$.  Then
\[ \chi(\Kk) = c_1 + c_2 + c_3 - b.\]
\end{lemma}
\begin{proof}
If $\Kk$ is in bridge position, the components of $\Kk$ induce a cell decomposition with $2b$ 0-cells, $3b$ 1-cells, and $c_1+ c_2 + c_3$ 2-cells.
\end{proof}

As with trisections, we may wish to assign a complexity to generalized bridge trisections.  The most specific designation has eight parameters.  If $\Tt$ is a generalized bridge trisection, and the underlying trisection has complexity $(g;k_1,k_2,k_3)$, we say that the complexity of the generalized bridge trisection is $(g;k_1,k_2,k_3;b;c_1,c_2,c_3)$.  In the case that $k = k_1 = k_2 = k_3$ and $c = c_1 = c_2 = c_3$, we say that $\Tt$ is \emph{balanced} and denote its complexity by $(g,k,b,c)$.  Even more generally, a $(g,b)$--generalized bridge trisection refers to a generalized bridge trisection with a genus $g$ central surface that meets $\Kk$ in $2b$ points.  In Section~\ref{sec:exs}, we classify all $(1,1)$--generalized bridge trisections.  If the underlying trisection is the genus zero trisection of $S^4$, as in~\cite{Meier-Zupan_Bridge_2017}, we call $\Tt$ a $(b;c_1,c_2,c_3)$--bridge trisection or a $(b,c)$--bridge trisection in the balanced case.

\section{Examples}\label{sec:exs}

Before including proofs in the next section, we present several examples of generalized bridge trisections and shadow diagrams of knotted surfaces in 4--manifolds.

\subsection{Shadow diagrams}\label{subsec:shadow}\ 


Just as a trisection diagram determines the spine of a trisection, a type of diagram called a \emph{tri-plane diagram} determines the spine of a bridge trisection, as shown in~\cite{Meier-Zupan_Bridge_2017}.  In dimension three, it is generally more difficult to draw diagrams for knots in manifolds other than $S^3$, and unfortunately tri-plane diagrams do not naturally extend from bridge trisections to generalized bridge trisections.  Instead, we employ a new structure called a \emph{shadow diagram}.  Let $\tau$ be a trivial tangle in a handlebody $H$.  A \emph{curve-and-arc system} $(\A,a)$ determining $(H,\tau)$ is a collection of pairwise disjoint simple closed curves $\A$ and arcs $a$ in $\Sigma = \pd H$ such that $\A$ determines $H$ and $a$ is a collection of shadow arcs for $\tau$.  Note that curves in $\A$ and arcs in $a$ can be chosen to be disjoint by standard cut-and-paste arguments using compressing disks for $H$ and bridge disks for $\tau$.  A \emph{shadow diagram} for a generalized bridge trisection $\Tt$ is a triple $((\A,a), (\n,b), (\g,c))$ of curve-and-arc systems determining the spine $(H_{31},\tau_{31}) \cup (H_{12},\tau_{12}) \cup (H_{23},\tau_{23})$ of $\Tt$. Since every trivial tangle in a handlebody can be defined by a curve-and-arc system, it is clear that Corollary~\ref{maincor} follows immediately from Theorem~\ref{mainthm}.

Of course, there are infinitely many different shadow diagrams corresponding to the same generalized bridge trisection, but these diagrams may be related by moves which we call disk-slides.  
We define a \emph{compressing disk} for a trivial tangle $(H,\tau)$ to be a properly embedded disk $D$ in $H \setminus \tau$ such that $\pd D \subset \pd H$ and $D$ is not boundary parallel in $H\setminus\tau$.  Suppose that $a_1$ is a curve that bounds a compressing disk for $(H,\tau)$ or a shadow arc for some $\tau_1 \in \tau$, and let $a_2$ be a curve that bounds a compressing disk for $(H,\tau)$.
A curve or arc $a_1'$ created by banding $a_1$ to $a_2$ along an embedded arc in $\Sigma$ is said to be the result of a \emph{disk-slide} of $a_1$ over $a_2$.
It is straightforward to see if $a_1'$ is the result of a disk-slide of $a_1$ over $a_2$, then $a_1'$ bounds a compressing disk or is the shadow of the trivial arc $\tau_1$ in $(H,\tau)$.

\begin{proposition}
	Any two shadow diagrams for a fixed generalized bridge trisection $\Tt$ are related by a sequence of disk-slides within the respective curve-and-arc systems.
\end{proposition}

\begin{proof}
The proof follows from work of Johannson in the case that $\tau = \emp$~\cite{Johannson_Topology_1995}.  In general, it follows from a standard cut-and-paste argument.  A detailed proof is left as an exercise for the reader.
\end{proof}

\subsection{1--bridge trisections}\ 

One family which deserves special consideration is the collection of 1-bridge trisections, i.e., $(g,1)$--generalized bridge trisections.  If $\Kk$ has a such a splitting, then it intersects each sector $X_i$ of the underlying trisection in a single disk and each handlebody in a single arc.  In this case, we deduce from Lemma~\ref{echar} that $\Kk$ is a 2--knot, and the generalized bridge trisection is efficient.  Shadow diagrams for generalized bridge trisections of this type are particularly simple: A \emph{doubly-pointed trisection diagram} is a shadow diagram in which each curve-and-arc system contains exactly one arc. In this case, drawing the arc in the diagram is redundant, since there is a unique way (up to disk-slides) to connect the two points in the complement of any one of the sets of curves.

In Figure~\ref{fig:1bDiags}, we depict several doubly pointed diagrams for low-complexity examples.  First, we give diagrams for the two simplest complex curves in $\CP^2$, namely, the line $\CP^1$ and the quadric $\Cc_2$.  Next, we give diagrams for an $S^2$--fiber in $S^2\times S^2$ and the sphere $\Cc_{(1,1)}$ in $S^2 \wt{\times} S^2$ representing $(1,1)\in \Z\oplus\Z\cong H_2(S^2\wt{\times}S^2)$.   (See~\cite{Gompf-Stipsicz_4-manifolds_1999} for formal definitions.)  We postpone the justification for these diagrams until Section~\ref{sec:proofs}, in which we develop the machinery to make such justification possible.  

\begin{figure}[h!]
	\centering
	\includegraphics[width=.95\textwidth]{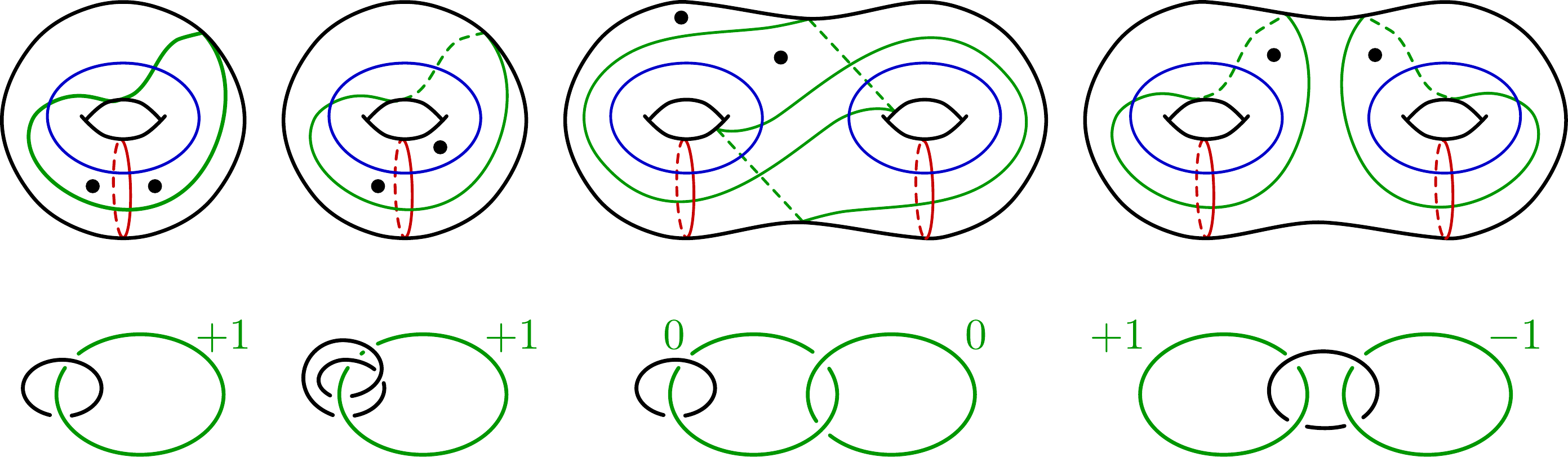}
	\caption{Some doubly-pointed trisection diagrams.  From left to right: $(\CP^2,\CP^1)$, $(\CP^2,\Cc_2)$, $(S^2\times S^2, S^2\times\{\ast\})$, and $(S^2\wt{\times} S^2, \Cc_{(1,1)})$. }
	\label{fig:1bDiags}
\end{figure}

By results in~\cite{Meier-Schirmer-Zupan_Classification_2016,Meier-Zupan_Bridge_2017,Meier-Zupan_Genus-two_2017}, any surface with a $(0,b)$--generalized bridge trisection (i.e. a $b$--bridge trisection) for $b < 4$ is unknotted in $S^4$.  In Section~\ref{sec:proofs}, we will prove the following classification result.

\begin{proposition}\label{oneone}\ 
\begin{enumerate}
	\item There are exactly two nontrivial $(1,1)$--knots (up to change of orientation and mirroring): $(\CP^2,\CP^1)$ and $(\CP^2,\Cc_2)$.
	\item Other than the two examples from (1), there is no nontrivial 2--knot in either $S^4$, $S^1\times S^3$, or $\CP^2$ that admits a $(2,1)$--generalized bridge trisection.
\end{enumerate}
\end{proposition}

On the other hand, there are many 2--knots admitting $(3,1)$--generalized bridge trisections:  Perform three meridional stabilizations (defined in Section~\ref{sec:proofs}) on any $(4,2)$--bridge trisection, of which there are infinitely many~\cite{Meier-Zupan_Bridge_2017}.  We offer the following as worthwhile problems.

\begin{problems}\label{prob:class}\ 
	\begin{enumerate}
		\item Classify 2--knots admitting $(2,1)$--generalized bridge trisections.
		\item Classify projective planes admitting $(1,2)$--generalized bridge trisections.
	\end{enumerate}
\end{problems}

With regard to (2) of Problem~\ref{prob:class}, Figure~\ref{fig:proj} shows a $(2,1)$--shadow diagram for the standard projective (real) plane $(\CP^2, \RP^2)$ that is the lift  of the standard cross-cap in $S^4$ under the branched double covering.  More generally, consider a surface knot (or link) $(X,\Kk)$, and let $X_n(\Kk)$ denote the $n$--fold cover of $X$, branched along $\Kk$.  Let $\widetilde\Kk_n$ denote the lift of $\Kk$ under this covering.

\begin{figure}[h!]
	\centering
	\includegraphics[width=.95\textwidth]{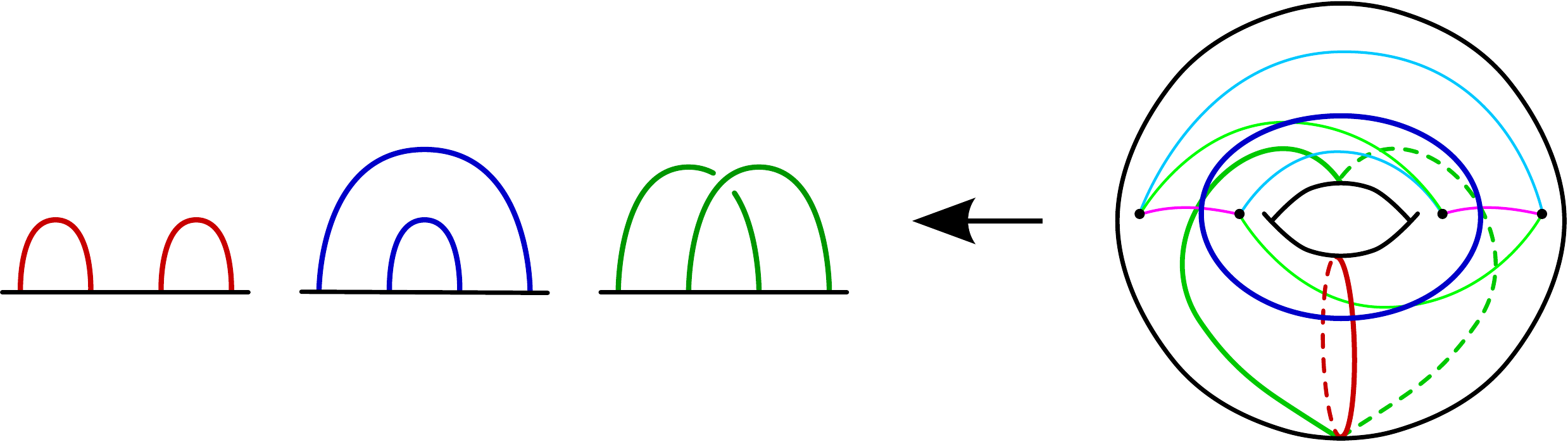}
	\caption{The branched double covering projection relating the standard cross-cap $(S^4,\Pp_+)$ and its cover $(\CP^2,\RP^2)$. }
	\label{fig:proj}
\end{figure}

\begin{proposition}\label{prop:branch}
	If $(X, \Kk)$ admits a $(g;k_1,k_2,k_3;b;c_1,c_2,c_3)$--generalized bridge trisection, then $(X_n(\Kk),\widetilde\Kk)$ admits a $(g';k_1',k_2',k_3';b,c_1,c_2,c_3)$--generalized bridge trisection, where $g' = ng+(n-1)(b-1)$ and $k_i'=nk_i+(n-1)(c_i-1)$.
\end{proposition}

\begin{proof}
	It suffices to show that the $n$--fold cover of a genus $g$ handlebody branched along collection of $b$ trivial arcs is a handlebody of genus $g' = ng + (n-1)(b-1)$, with the lift of the original $b$ trivial arcs being a collection of $b$ trivial arcs upstairs.  From this, the rest of the proposition follows, once we observe that the trivial disk system $(\natural^k(S^1\times B^3),\Dd)$ is simply the trivial tangle product $(\natural^k(S^1\times D^2),\tau)\times I$, and that the branched covering respects this product structure.  Thus, each piece of the trisection lifts to a standard piece, so the cover is trisected.
	
	To form the $n$--fold branched cover of $(H,\tau)$, we cut $H$ open along a collection of bridge disks $\{\Delta_i\}$ for $\tau = \{\tau_i\}$.  We then take $n$ copies of this filleted object, and glue them together cyclically by identifying  $\Delta_i^+$ of one copy with $\Delta_i^-$ of the next.  This is the same as attaching $nb$ 1--handles to the disjoint union of $n$ genus $g$ handlebodies until the object is connected.  The result is a genus $ng+(n-1)(b-1)$ handlebody.  Clearly, one of the lifts of $\Delta_i$ is a bridge disk for the lift of $\tau_i$, as desired.
\end{proof}

	

\subsection{Complex curves in $\CP^2$}\label{subsec:complex}\ 

In this subsection we summarize results that have been obtained by the authors in collaboration with Peter Lambert-Cole regarding generalized bridge trisections of complex curves in complex 4--manifolds of low trisection genus (e.g., $\CP^2$, $S^2\times S^2$, and $\CP^2\#\overline{\CP^2}$).  The following is a preliminary result that will be expanded in future joint work with Lambert-Cole. Let $\Cc_d$ denote the complex curve of degree $d$ in $\CP^2$.  Note that $\Cc_d$ is a closed surface of genus $(d-1)(d-2)/2$.

\begin{reptheorem}{thm:complex}
	The pair $(\CP^2,\Cc_d)$ admits a $(1,1;(d-1)(d-2)+1,1)$--generalized bridge trisection.
\end{reptheorem}

In other words, complex curves in $\CP^2$ admit efficient generalized bridge trisections with respect to the genus one trisection of $\CP^2$; each such curve can be decomposed as the union of three disks ($c=1$).  See Figure~\ref{fig:cubic}.  Let $X_{n,d}$ denote the 4--manifold obtained as the $n$--fold cover of $\CP^2$, branched along $\Cc_d$, which exists whenever $n$ divides $d$.  When $n=d$, we have that $X_{d,d}$ is the degree $d$ hyper-surface in $\CP^3$.  The next corollary follows from Theorem~\ref{thm:complex} and Proposition~\ref{prop:branch}.
\begin{figure}[h!]
	\centering
	\includegraphics[width=.8\textwidth]{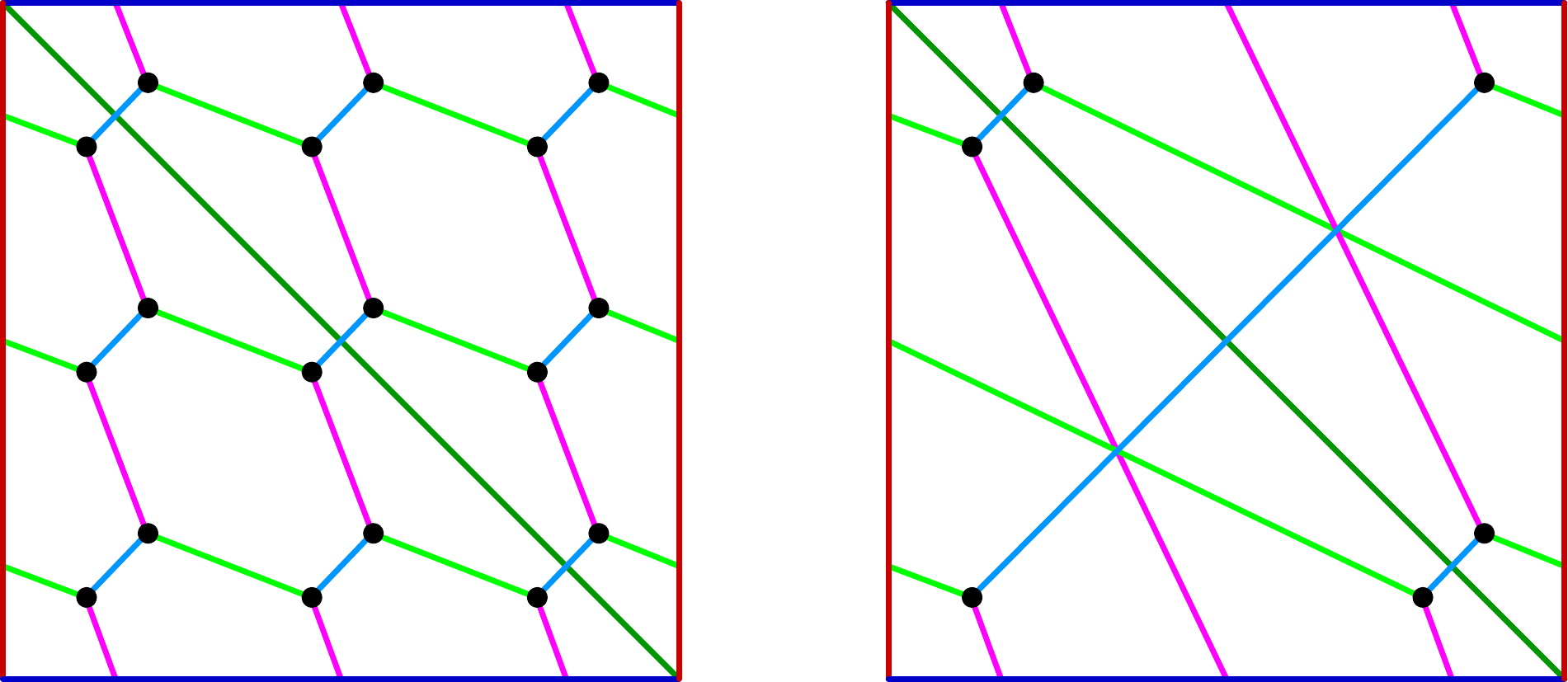}
	\caption{Two shadow diagrams for $\Cc_3$ in $\CP^2$. The diagram on the left is due to Peter Lambert-Cole, and the diagram on the right is efficient.}
	\label{fig:cubic}
\end{figure}

\begin{corollary}\label{coro:g0}
	$X_{n,d}$ admits an efficient $(g,0)$--trisection where $g=n+(n-1)(d-1)(d-2)$.
\end{corollary}

Note that $Z_{p,q,r}= p\CP^2\# q\overline{\CP^2}\# r S^2\times S^2$ admits as $(g,0)$--trisection where $g = p+q+2r$.  It had been speculated that an extension of the main theorems of~\cite{Meier-Schirmer-Zupan_Classification_2016} and~\cite{Meier-Zupan_Genus-two_2017} would show that every manifold admitting a $(g,0)$--trisection is diffeomorphic to $Z_{p,q,r}$; however, Corollary~\ref{coro:g0} gives many interesting counterexamples to this suspicion.

For example, if $d$ is odd and at least five, then $X_{d,d}$ is homeomorphic, but not diffeomorphic, to $Z_{p,q,0}$ for certain $p,q\geq 0$~\cite{Don_1990_Polynomial-invariants}.  Thus, we see that there are pairs of exotic manifolds that are not distinguished by their trisection invariants.  We note that Baykur and Saeki have previously given examples of inefficient exotic trisections~\cite{Baykur-Saeki_Simplifying_2017}.  Justification of the diagrams in Figure~\ref{fig:cubic}, along with the complete proof of Theorem~\ref{thm:complex} will appear in a forthcoming article based on the aforementioned collaboration with Lambert-Cole.

\section{Proofs}\label{sec:proofs}

In the first part of this section, we will prove a sequence of lemmas which, taken together, imply Theorem~\ref{mainthm}.  In the second part, we prove Proposition~\ref{oneone}, classifying $(1,1)$--generalized bridge trisections.  In the third part, we introduce the notion of meridional stabilization and prove Theorem~\ref{mainthm2}.

\subsection{The existence of generalized bridge splittings}\label{subsec:exist}\ 

Here we discuss the interaction between handle decompositions and trisections of closed 4--manifolds.  We will not rigorously define handle decompositions but direct the interested reader to~\cite{Gompf-Stipsicz_4-manifolds_1999}.  

Suppose $\Hh$ is a handle decomposition of a 4--manifold $X$ with a single 0--handle and a single 4--handle.  Corresponding to $\Hh$, there is a Morse function $h: X \rightarrow \R$, equipped with a gradient-like vector field, that induces the handle decomposition $\Hh$.  We will suppose that each Morse function is equipped with a gradient-like vector field, which we will neglect to mention henceforth.  After an isotopy, we may assume that every critical point of index $i$ occurs in the level $h^{-1}(i)$.  Such a Morse function is called \emph{self-indexing}.  For any subset $S \subset \R$, let $Y_S$ denote $X \cap h^{-1}(S)$.  Let $Z$ be a compact submanifold of $Y_{\{t\}}$ for some $t$, and let $[r,s]$ be an interval containing $t$.  We will let $Z_{[r,s]}$ denote the subset of $X$ obtained by pushing $Z$ along the flow of $h$ during time $[s,r]$.  In particular, if this set does not contain a critical point of $h$, then $Z_{[r,s]}$ is diffeomorphic to $Z \X [r,s]$, and we let $Z_{\{t'\}}$ denote $Z_{[r,s]} \cap h^{-1}(t')$.

Now, let $\Hh$ be a handle decomposition of $X$ with $n_i$ $i$--handles for $i = 1,2,3$, and let $T$ be the attaching link for the 2--handles, so that $T$ is an $n_2$-component framed link contained in $\#^{n_1}(S^1 \X S^2)$ with Dehn surgery yielding $\#^{n_3} (S^1 \X S^2)$.  In addition, let $h: X \rightarrow \R$ be a self-indexing Morse function inducing $\Hh$.  We suppose without loss of generality that $T$ is contained in $Y_{\{3/2\}} = \#^{n_1} (S^1 \X S^2)$, and we let $\Sigma$ be a genus $g$ Heegaard surface cutting $Y_{\{3/2\}}$ into handlebodies $H$ and $H'$, where a core of the handlebody $H$ contains $T$.

\begin{lemma}\label{morsetri}
Let $X$ be a 4--manifold with self-indexing Morse function $h$, and, using the notation above, consider the sets
\[ X_1 = Y_{[0,3/2]} \cup H'_{[3/2,2]}, \qquad X_2 = H_{[3/2,5/2]}, \qquad X_3 = H'_{[2,5/2]} \cup Y_{[5/2,4]}.\]
The decomposition $X = X_1 \cup X_2 \cup X_3$ is a $(g;n_1,g-n_2,n_3)$--trisection with central surface $\Sigma_{\{2\}}$.
\end{lemma}
\begin{proof}
This is proved in slightly different terms in~\cite{Gay-Kirby_Trisecting_2016}; we will include a brief proof here.  As mentioned above, a self-indexing Morse function gives rise to a handle decomposition $\Hh$ of $X$.  Clearly, $Y_{[0,3/2]}$ is a handlebody, as it contains the 0--handle and the $n_1$ 1--handles of $\Hh$, and $H'_{[3/2,2]}$ contains no critical points and thus is diffeomorphic to $H' \X [3/2,2]$.  Thus, the union $X_1$ is a genus $n_1$ 1--handlebody.  Similarly, $X_3$ contains the 3--handles and the 4--handle and is a genus $n_3$ 1--handlebody.  Finally, note that $X_2$ may be constructed by attaching 2--handles to $H \X [3/2,2-\varepsilon]$ along $T$, where $H \X [3/2,2-\varepsilon]$ is built from a 0--handle and $g$ 1--handles, and each 2--handle attached along $T$ cancels a 1--handle, so that $X_2$ has a handle decomposition with a 0--handle and $g-n_2$ 1--handles, making it a genus $g-n_2$ 1--handlebody.

Next, we describe the pairwise intersections of $X_i$'s.  We have $X_1 \cap X_2 = H_{\{3/2\}} \cup \Sigma_{[3/2,2]}$, $X_2 \cap X_3 = H_{{\{5/2\}}} \cup \Sigma_{[2,5/2]}$, and $X_1 \cap X_3 = H'_{\{2\}}$, where each intersection is a genus $g$ handlebody.  Finally, note that $X_1 \cap X_2 \cap X_3 = \Sigma_{\{2\}}$, as desired.
\end{proof}
Given a self-indexing Morse function $h$ and surface $\Sigma$ as above, we will let $\Tt(h,\Sigma)$ denote the trisection described by Lemma~\ref{morsetri}.

\begin{lemma}\label{trimorse}
Given a trisection $\Tt$ of $X$, there is a self-indexing Morse function $h$ and surface $\Sigma \subset Y_{\{3/2\}}$ such that $\Tt = \Tt(h,\Sigma)$.
\end{lemma}
\begin{proof}
This is essentially contained in the proof of Lemma 13 of~\cite{Gay-Kirby_Trisecting_2016}.
\end{proof}

Now we turn our focus to knotted surfaces in 4--manifolds.  Suppose that $\Kk$ is a knotted surface in $X$.  A \emph{Morse function of the pair} $h:(X,\Kk) \rightarrow \R$ is a Morse function $h: X \rightarrow \R$ with the property that the restriction $h_{\Kk}$ is also Morse.  Note that for any $\Kk \subset X$, a Morse function $h:X \rightarrow \R$ becomes a Morse function of the pair $(X,\Kk)$ after a slight perturbation of $\Kk$ in $X$.  Expanding upon the previous notation, for $S \subset \R$, we let $L_S$ denote $\Kk \cap h^{-1}(S)$.  Let $J$ be a compact submanifold of $L_{\{t\}}$ for some $t$, and let $[r,s]$ be an interval containing $t$.  We will let $J_{[r,s]}$ denote the subset of $\Kk$ obtained by pushing $J$ along the flow of $h_{\Kk}$ during time $[s,r]$.  As above, if this set does not contain a critical point of $h_{\Kk}$, then $J_{[r,s]}$ is diffeomorphic to $J \X [r,s]$.

Saddle points of $h_{\Kk}$ can be described as cobordisms between links obtained by resolving bands:  Given a link $L$ in a 3--manifold $Y$, a \emph{band} is an embedded rectangle $R = I \X I$ such that $R \cap L = \pd I \X I$.  We \emph{resolve the band $R$} to get a new link by removing the arcs $\pd I \X I$ from $L$ and replacing them with the arcs $I \X \pd I$.  Note that every band $R$ can be represented by a framed arc $\eta = I \X \{1/2\}$, so $\eta$ meets $L$ only in its endpoints.  Let $h$ be a Morse function of the pair $(X,\Kk)$, suppose that all critical points of $h$ and $h_{\Kk}$ occur at distinct levels, and let $x \in \Kk$ be a saddle point contained in the level $h^{-1}(t)$.  Then there is a framed arc $\eta$ with endpoints in the link $L_{\{t-\eps\}}$ with the property that the link $L_{\{t+\eps\}}$ is obtained from $L_{\{t-\eps\}}$ by resolving the band corresponding to $\eta$.  We will use this fact in the proof of the next lemma, which is related to the notion of a normal form for a 2--knot in $S^4$~\cite{Kawauchi-Shibuya-Suzuki_Descriptions_1982,Kawauchi_A-survey_1996}.

\begin{lemma}\label{doublemorse}
Suppose $X$ is a 4--manifold equipped with a handle decomposition $\Hh$, and $\Kk$ is a surface embedded in $X$.  After an isotopy of $\Kk$, there exists a Morse function of the pair $(X,\Kk)$ such that
\begin{enumerate}
\item $h$ is a self-indexing Morse function inducing the handle decomposition $\Hh$,
\item minima of $h_{\Kk}$ occur in the level $Y_{\{1\}}$,
\item saddles of $h_{\Kk}$ occur in the level $Y_{\{2\}}$, and
\item maxima of $h_{\Kk}$ occur in the level $Y_{\{3\}}$.
\end{enumerate}
\end{lemma}
\begin{proof}
Let $\Gamma_1$ be an embedded wedge of circles containing the cores of the 1--handles, so that $\nu(\Gamma_1)$ is the union of the 0--handle and the 1--handles of $\Hh$.  Similarly, let $\Gamma_3$ be an embedded wedge of circles such that $\nu(\Gamma_3)$ is the union of the 3--handles and 4--handle.  After isotopy $\Kk$ meets $\Gamma_1$ and $\Gamma_3$ transversely; hence $\Kk \cap \Gamma_1 = \Kk \cap \Gamma_3 = \emp$, and thus we can initially choose a self-indexing Morse function $h:X \rightarrow \R$ so that $\nu(\Gamma_1) = Y_{[0,1+\eps)}$, $\nu(\Gamma_3) = Y_{(3-\eps,4]}$, and $\Kk \subset Y_{(1+\eps,3-\eps)}$.  For each minimum point of $h_{\Kk}$, choose a descending arc avoiding $\Kk$ and the critical points of $h$ and drag the minimum downwards within a neighborhood of this arc until it is contained in $Y_{\{1\}}$.  Similarly, there is an isotopy of $\Kk$ after which all maxima are contained in $Y_{\{3\}}$.

It only remains to show that after isotopy, all saddles of $h_{\Kk}$ are contained in $Y_{\{2\}}$.  Let $T$ be the attaching link for the 2--handles of $\Hh$, considered as a link in $Y_{\{2\}}$.  For each saddle point $x_i$ in level $t_i < 2$, let $\eta_i$ be the framed arc with endpoints in $L_{\{t_i\}}$, where $1 \leq i \leq n$, so that $L_{t_i + \eps}$ is obtained from $L_{t_i - \eps}$ by resolving the band induced by $\eta_i$.  Certainly $\eta_1$ is disjoint from $L_{t_1-\eps}$ except at its endpoints.  A priori, $\eta_2$ may intersect the band induced by $\eta_1$, but after a small isotopy, we may assume that $\eta_2$ avoids $\eta_1$ and thus we can push $\eta_2$ into $Y_{\{t_1\}}$.  Continuing this process, we may push all arcs $\eta_i$ into $Y_{\{t_1\}}$, and generically, the graph $L_{t_1 - \eps} \cup \{\eta_i\}$ is disjoint from $T$, so the entire apparatus can be pushed into $Y_{\{2\}}$.  A parallel argument show that the framed arcs coming from saddles occurring between $t = 2$ and $t=3$ can be pushed down into $Y_{\{2\}}$, as desired.
\end{proof}

We call a Morse function $h:(X,\Kk) \rightarrow \R$ that satisfies the conditions in Lemma~\ref{doublemorse} a \emph{self-indexing Morse function of the pair $(X,\Kk)$}.  Given such a function, we can push the framed arcs $\{\eta_i\}$ corresponding to the saddles of $h_{\Kk}$ into the level $Y_{\{3/2\}}$, where the endpoints of $\{\eta_i\}$ are contained in $L_{\{3/2\}}$ and resolving $L_{\{3/2\}}$ along the bands given by $\{\eta_i\}$ yields the link $L_{\{5/2\}}$.  A \emph{banded link diagram} for $\Kk$ consists of the union of $L_{\{3/2\}}$ with the bands given by $\{\eta_i\}$, contained in $Y_{\{3/2\}}$, along with the framed attaching link for the 2--handles in $X$, denoted by $T \subset Y_{\{3/2\}}$.  As such, a banded link diagram completely determines the knotted surface $\Kk \subset X$.  Let $\Hh$ be the handle decomposition of $X$ determined by $h$.  As above, let $\Sigma$ be a Heegaard surface cutting $Y_{\{3/2\}}$ into handlebodies $H$ and $H'$, where a core of $H$ contains $T$.

Let $\Gamma = L_{\{3/2\}} \cup \{\eta_i\}$ in $Y_{\{3/2\}}$.  We will show that $\Gamma$ may be isotoped to be in a relatively nice position with respect to the surface $\Sigma$, from which it will follow that there is an isotopy of $\Kk$ to be in a relatively nice position with respect to the trisection $\Tt(h,\Sigma)$.  An arc $\eta \subset \pd H$ is \emph{dual} to a trivial arc $\tau_i \subset H$ if there is a shadow $\tau_i'$ for $\tau_i$ that meets $\eta$ in one endpoint. Finally, a collection of pairwise disjoint arcs $\{\eta_i\} \subset \pd H$ is said to be \emph{dual} to a trivial tangle $\{\tau_i\}$ if there is a collection of shadows $\{\tau'_i\}$ that meet $\{\eta_i\}$ only in their endpoints and such that each component of $\{\eta_i\} \cup \{\tau'_i\}$ is simply connected (in other words, this collection contains only arcs, not loops).

We say that $\Gamma$ is in \emph{bridge position} with respect to $\Sigma$ if
\begin{enumerate}
\item each of $L_{\{3/2\}} \cap H$ and $L_{\{3/2\}} \cap H'$ is a trivial tangle,
\item $\{\eta_i\} \subset \Sigma$ with framing given by the surface framing, and
\item the arcs $\{\eta_i\}$ are dual to the trivial arcs $L_{\{3/2\}} \cap H$.
\end{enumerate}

To clarify condition (2), the arc $\eta_i \subset \Sigma$ has framing given by the surface framing exactly when the band induced by $\eta_i$ meets $\Sigma$ in the single arc $\eta_i$.  Next, we show that such structures exist, after which we describe how they induce generalized bridge trisections of $(X,\Kk)$.

\begin{lemma}\label{bridgeposition}
Given a knotted surface $\Kk \subset X$ and a self-indexing Morse function $h$ of the pair $(X,\Kk)$, let $\Sigma$, $H$, $H'$, $T$, and $\Gamma$ be as defined above.  There exists an isotopy of $\Gamma$ in $Y_{\{3/2\}}$ after which $\Gamma$ is in bridge position with respect to $\Sigma$.
\end{lemma}
\begin{proof}
This decomposition is similar to the notion of a \emph{banded bridge splitting} from~\cite{Meier-Zupan_Bridge_2017}, where the detailed arguments in Theorem 1.3 do not make use of the fact that $\Sigma$ is sphere and thus transfer directly to this setting.  We give a brief outline of the proof but refer the reader to~\cite{Meier-Zupan_Bridge_2017} for further details.

Consider cores $C \subset H$ and $C' \subset H'$, which may be chosen so that $T \subset C$ and both $C$ and $C'$ are disjoint from $\Gamma$.  Note that $Y_{\{3/2\}} \setminus (C \cup C')$ is diffeomorphic to $\Sigma \X (-1,1)$ and thus there is a natural projection from $Y_{\{3/2\}} \setminus (C \cup C')$ onto $\Sigma = \Sigma \X \{0\}$.  By equipping this projection with crossing information, we may view it as an isotopy of $\Gamma$ within $Y_{\{3/2\}} \setminus (C \cup C')$.  First, if the arcs $\{\eta_i\}$ project to arcs that cross themselves or each other, we may stretch $L_{\{3/2\}}$ and shrink $\{\eta_i\}$ so these crossings are slid to $L_{\{3/2\}}$, after which the projection of the collection $\{\eta_i\}$ is embedded in $\Sigma$. (See Figure~10 of~\cite{Meier-Zupan_Bridge_2017}.)  It may be possible that some surface framing of some arc $\eta_i$ disagrees with its given framing; in this case, an isotopy of $L_{\{3/2\}}$ allows $\eta_i$ to be pushed off of and back onto $\Sigma$ with the desired framing, as in Figure~11 of~\cite{Meier-Zupan_Bridge_2017}.  Thus, we may assume condition (2) of the definition of bridge position of $\Gamma$ is satisfied.

Now, we push the projection $L_{\{3/2\}}$ off of $\Sigma$ so that $L_{\{3/2\}}$ is in bridge position, fulfilling condition (1) of the definition of bridge position.  At this point, it may not be the case that the arcs $\{\eta_i\}$ are dual to $L_{\{3/2\}} \cap H$; however, this requirement may be achieved by perturbing $L_{\{3/2\}}$ near the endpoints of the arcs $\{\eta_i\}$ in $\Sigma$, as in Figure~12 of~\cite{Meier-Zupan_Bridge_2017}.
\end{proof} 

\begin{lemma}\label{bridgetri}
Suppose that $\Kk$ is a knotted surface in $X$, with self-indexing Morse function $h$ of the pair $(X,\Kk)$ and $\Sigma$, $H$, $H'$, $T$, and $\Gamma$ as defined above.  Suppose further that $\Gamma$ is in bridge position with respect to $\Sigma$, push the arcs $\{\eta_i\}$ slightly into the interior of $H$.  Let $X_1$, $X_2$, and $X_3$ be defined as in Lemma~\ref{morsetri}, and define $\Dd_i = \Kk \cap X_i$.  Then
\[ (X,\Kk) = (X_1,\Dd_1) \cup (X_2, \Dd_2) \cup (X_3,\Dd_3)\]
is a generalized bridge trisection of $(X,\Kk)$.
\end{lemma}
\begin{proof}
By Lemma~\ref{morsetri}, the underlying decomposition $X = X_1 \cup X_2 \cup X_3$ is a trisection, and thus we must show that $\Dd_i$ is a trivial disk system in $X_i$ and $\Dd_i \cap \Dd_j$ is a trivial tangle in the handlebody $X_i \cap X_j$.

Let $\tau = L_{\{3/2\}} \cap H$ and $\tau' = L_{\{3/2\}} \cap H'$, so that each of $\tau$ and $\tau'$ is a trivial tangle in $H$ and $H'$, respectively.  We note that by construction, $\Dd_1 = L_{[1,3/2]} \cup \tau'_{[3/2,2]}$, $\Dd_2 = \tau_{[3/2,5/2]}$, and $\Dd_3 = L_{[5/2,3]} \cup \tau'_{[2,5/2]}$.  Thus, there is a Morse function of the pair $(X_1,\Dd_1)$ that contains only minima, so that $\Dd_1$ is a collection of trivial disks in $X_1$.  Similarly, $(X_3,\Dd_3)$ contains only maxima, so that $\Dd_3 \subset X_3$ is a collection of trivial disks as well.  We also note that $\Dd_1 \cap \Dd_3 = \tau'_{\{2\}}$, a collection of trivial arcs in $X_1 \cap X_3$, and $\Dd_1 \cap \Dd_2 = \tau_{\{3/2\}} \cup (\pd \tau)_{[3/2,2]}$, a collection of trivial arcs in $X_1 \cap X_2$.

It only remains to show that $\Dd_2$ is a collection of trivial disks in $X_2$, and $\Dd_2 \cap \Dd_3$ is a collection of trivial arcs in $X_2 \cap X_3$.  However, this follows immediately from Lemma 3.1 of~\cite{Meier-Zupan_Bridge_2017}; although the proof of Lemma 3.1 is carried out in the context of the standard trisection of $S^4$, it can be applied verbatim here.
\end{proof}

\begin{proof}[Proof of Theorem~\ref{mainthm}]
By Lemma~\ref{trimorse}, there exists a self-indexing Morse function $h:X \rightarrow \R$ and Heegaard surface $\Sigma \subset Y_{\{3/2\}}$ such that $\Tt = \Tt(h,\Sigma)$.  Applying Lemma~\ref{doublemorse}, we have that there is an isotopy of $\Kk$ after which $h:(X,\Kk) \rightarrow \R$ is a self-indexing Morse function of the pair.  Moreover, by Lemma~\ref{bridgeposition}, there is a further isotopy of $\Kk$ after which the graph $\Gamma$ induced by the saddle points of $h_{\Kk}$ is in bridge position with respect to $\Sigma$.  Finally, the decomposition defined in Lemma~\ref{bridgetri} is a generalized bridge trisection of $(X,\Kk)$, completing the proof.
\end{proof}

We note that as in Lemma 3.3 and Remark 3.4 from~\cite{Meier-Zupan_Bridge_2017}, this process is reversible; in other words, every bridge trisection of $(X,\Kk)$ can be used to extract a handle decomposition of $\Kk$ within $X$.  The proof of Lemma 3.3 applies directly in this case, and when we combine it with Lemma~\ref{trimorse} above, we have the following:

\begin{proposition}\label{handlestuff}
If $\Tt$ is a $(g;k_1,k_2,k_3;b;c_1,c_2,c_3)$--generalized bridge trisection of $(X,\Kk)$, then there is a Morse function $h$ of the pair $(X,\Kk)$ such that
{\begin{enumerate}
\item $h$ has $k_1$ index one critical points, $g - k_2$ index two critical points, and $k_3$ index three critical points; and
\item $h_{\Kk}$ has $c_1$ minima, $b- c_2$ saddles, and $c_3$ maxima.
\end{enumerate}}
\end{proposition}

We can now justify the diagrams in Figure~\ref{fig:1bDiags}.  By Proposition~\ref{handlestuff}, a 1--bridge trisection will give rise to a banded link diagram without bands corresponding to a Morse function $h$ of the pair $(X,\Kk)$ such that $h_{\Kk}$ has a single minimum and maximum.  From the shadow diagrams in Figure~\ref{fig:1bDiags}, we extract banded link diagrams, shown directly beneath each shadow diagram.    In each case, the black curve in Figure~\ref{fig:1bDiags} bounds a disk (the minimum of $h_{\Kk}$) in the 4--dimensional 0--handle, and a disk (the maximum of $h_{\Kk}$) in the union of the 2--handles with the 4--handle.  For example, in the first and third figure, we see that the 2--knot is the union of a trivial disk in the 0--handle, together with a cocore of a 2--handle.  The second figure is a well-known description of the quadric.  See Subsection~\ref{subsec:complex} above. The fourth figure can be obtained by connected summing the first figure with its mirror.

\subsection{Classification of $(1,1)$--generalized bridge trisections}\label{subsec:mstab}\ 

In this subsection, we prove Proposition~\ref{oneone}, classifying $(1,1)$--generalized bridge trisections.

\begin{proof}[Proof of Proposition~\ref{oneone}]

First, suppose that $(X,\Kk)$ admits a $(1,1)$--generalized bridge trisection $\Tt$.  Then $c_1 = c_2 = c_3 = 1$, $\chi(\Kk) = 2$, and $\Kk$ is a 2--sphere.  In addition, by Proposition~\ref{handlestuff}, there is a self-indexing Morse function $h$ on $(X,\Kk)$ so that $h_{\Kk}$ has one minimum, one maximum, and no saddles.  If $h$ has no index two critical points, then $(X,\Kk)$ is the double of a trivial disk in a 4--ball or 1--handlebody; thus $\Kk$ is unknotted.  If any one $k_i = 1$, then after permuting indices, we may assume that the induced $h$ has no index two critical points.  Thus, the only remaining case is $k_1 = k_2 = k_3 = 0$, and so $X = \CP^2$ or $\overline{\CP}^2$.

We will only consider the case $X = \CP^2$; parallel arguments apply by reversing orientations.  Let $h$ be a self-indexing Morse function for $\Tt$, so that $Y_{\{3/2\}}$ is diffeomorphic $S^3$, $L_{\{3/2\}}$ is an unknot we call $C$, and $T$ is a $(+1)$--framed unknot disjoint from $C$ in $Y_{\{3/2\}}$.  In addition, attaching a 2--handle to $T$ yields another copy of $S^3$, in which $C$ remains unknotted.  In other words, $C$ is an unknot in $S^3$ that is still unknotted after $(+1)$--Dehn surgery on $T$.  There are three obvious links $C \cup T$ that satisfy these requirements: a 2-component unlink, a Hopf link, and the torus link $T(2,2)$.  The first of these three corresponds to the unknotted 2--sphere.  The next two correspond to $\CP^2$ and $\Qq$, respectively.  We claim no other links $C \cup T$ of this type exist.
	
Consider $T$ as a (nontrivial) knot in the solid torus $S^3\setminus\nu(C)$.  Since $C$ remains unknotted after $(+1)$--surgery on $T$, it follows that $T$ is a knot in a solid torus with a solid torus surgery.  Let $\omega$ denote the linking number of $C$ and $T$, so that $\omega$ is also the winding number of $T$ in $S^3 \setminus \nu(C)$.  By~\cite{Gabai_1-bridge_1990}, one of the following holds:
\begin{enumerate}
\item $\omega = 1$ and $T \cup C$ is the Hopf link,
\item $\omega = 2$ and $T \cup C$ is the torus link $T(2,2)$, or
\item $\omega \geq 3$ and the slope of the surgery on $T$ is at least four.
\end{enumerate}
The third case contradicts the assumption that the surgery slope is one, completing the proof of part (1).

Part (2) follows from a similar argument.  Suppose $X = S^4$, $S^1 \times S^3$, $\CP^2$, or $\overline{\CP}^2$ and let $\Kk \subset X$ admit a $(2,1)$--generalized bridge trisection.  If some induced handle decomposition on $X$ has no 2--handles, then $\Kk$ is unknotted by the above arguments.  It follows from the genus two classification that the underlying trisection is either a $(2;1,1,0)$--trisection of $S^4$ or a $(2;1,0,0)$--trisection of $\CP^2$ or $\overline{\CP}^2$, up to permutation of the $k_i$'s.  The above argument applies directly to dispatch of the second case.  In the first case, we may assume the induced handle decomposition for $S^4$ is a canceling 2/3--pair, so that a banded link diagram for $\Kk$ consists of an unknot $C \subset S^3$ together with the attaching curve $T$ for the 2--handle, where $T$ is a 0--framed unknot.  In the complement of $C$, the curve $T$ is a knot in a solid torus with a surgery to $(S^1\times D^2)\#(S^1\times S^2)$.  By Scharlemann, $T$ is contained in a 3--ball, is unknotted, and is zero-framed~\cite{Scharlemann_Producing_1990}.  It follows that the 2--knot $\Kk$ is unknotted.
\end{proof}


\subsection{Meridional stabilization}\label{subsec:mstab}\ 

Consider a link $L$ in a 3--manifold $Y$, equipped with a $(g,b)$--bridge splitting $(Y,L) = (H_1,\tau_1) \cup (H_2,\tau_2)$, where $b \geq 2$.  Fix a trivial arc $\tau' \in \tau_2$, and let $H_1' = H_1 \cup \overline{\nu(\tau')}$ and $H_2' = H_2 \setminus \nu(\tau')$.  In addition, let $\tau_i' = L \cap H_i'$, so that $\tau_1' = \tau_1 \cup \tau'$ and $\tau_2' = \tau_2 \setminus \tau'$.  Then the decomposition $(Y,L) = (H_1',\tau_1') \cup (H_2',\tau_2')$ is a $(g+1,b-1)$--bridge splitting which is called a \emph{meridional stabilization} of the given $(g,b)$--splitting.  This construction is well-known.

In this subsection, we will extend meridional stabilization to a similar construction involving generalized bridge trisections in order to prove Theorem~\ref{mainthm2}.  Let $\Tt$ be a generalized bridge trisection for a connected knotted surface $\Kk \subset X$ with complexity $(g;k_1,k_2,k_3;b;c_1,c_2,c_3)$, and assume that $c_1\geq 2$.  Since $\Kk$ is connected, there exists an arc $\tau' \in \tau_{23}$ with the property that the two endpoints of $\tau'$ lie in different components of $\Dd_1$.  Define
\begin{enumerate}
	\item $(X_1',\Dd_1') = \left(X_1 \cup \overline{\nu(\tau')},\Dd_1 \cup \left(\overline{\nu(\tau')} \cap \Kk\right)\right)$ and 
	\item $(X_j',\Dd_j') = \left(X_j \setminus \nu(\tau'),\Dd_j \setminus \nu(\tau')\right)$ for $j=2,3$,
\end{enumerate}
and let $\Tt'$ be the decomposition
$$(X,\Kk) = (X_1',\Dd_1')\cup(X_2',\Dd_2')\cup(X_3',\Dd_3').$$
We say that the decomposition $\Tt'$ is obtained from $\Tt$ via \emph{meridional 1--stabilization along $\tau'$}.  We define meridional \emph{$i$--stabilization} similarly for $i =2$ or $3$.  Observe that the assumption that $\Kk$ is connected is slightly stronger than necessary; the existence of the arc $\tau' \in \tau_{jk}$ connecting two disks in $\Dd_i$ is necessary and sufficient.  Notably, $\Tt'$ is a generalized bridge splitting for $(X,\Kk)$, which we verify in the next lemma. Figure~\ref{fig:mstab} shows the local picture of a meridional 1--stabilization.

\begin{figure}[h!]
	\centering
	\includegraphics[width=.7\textwidth]{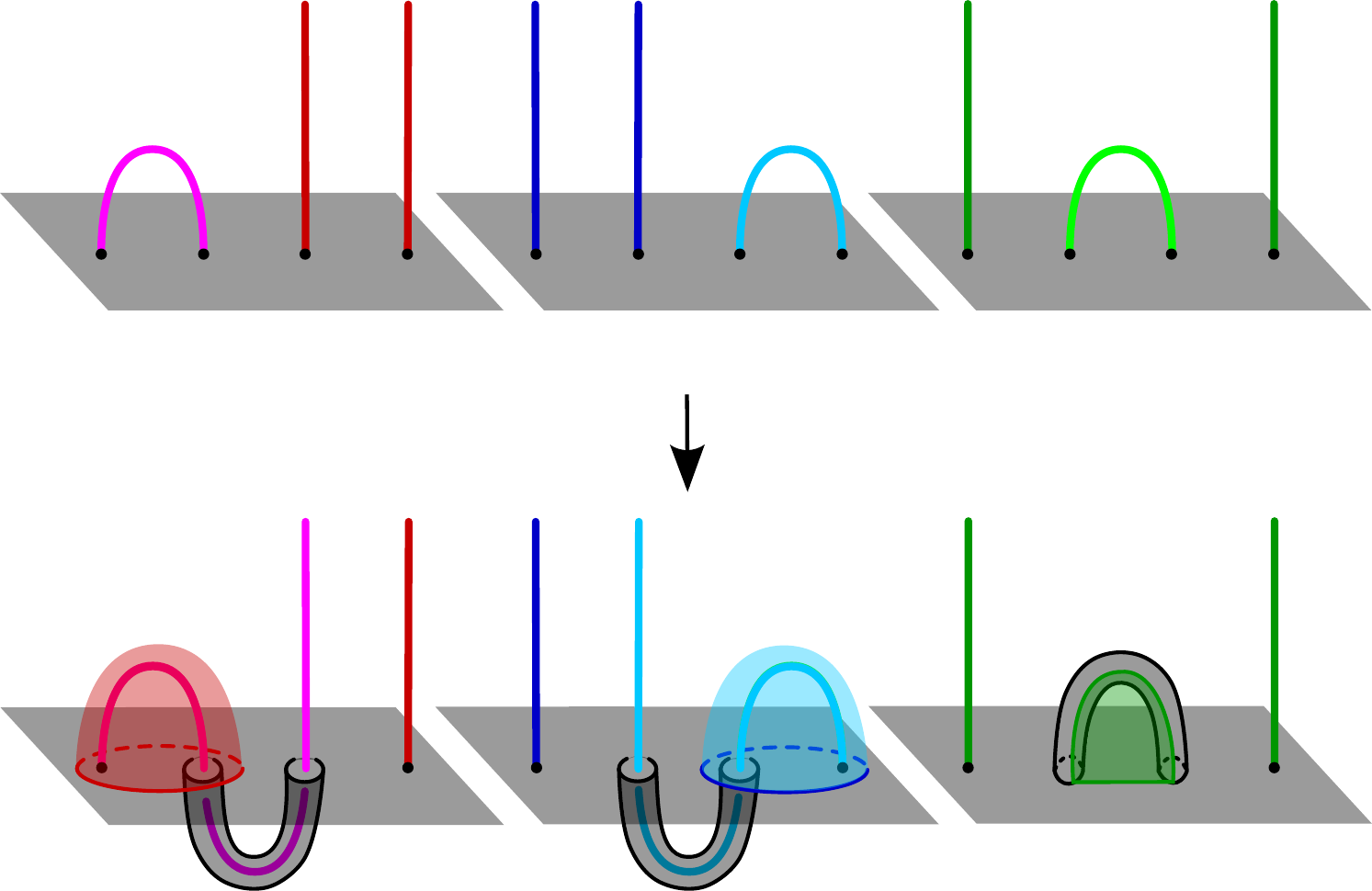}
	\caption{A sample meridional 1--stabilization along $\tau'$ (light green, top right).  Meridional stabilization increases the genus of the central surface by one, and a new compressing curve is shown for each handlebody in the bottom half of the figure.}
	\label{fig:mstab}
\end{figure}

\begin{lemma}\label{lem:mstab}
	The decomposition $\Tt'$ of $(X,\Kk)$ is a generalized bridge trisection of complexity $(g+1;k_1+1,k_2,k_3;b-1;c_1-1,c_2,c_3)$.
\end{lemma}

\begin{proof}
	Since $\tau'\subset \partial X_j$ for $j=2$ and $3$, we have that $(X_j',\Dd_j')\cong (X_j,\Dd_j)$.  Let $X' = \overline{\nu(\tau')} \cap (X_2 \cup X_3)$ and $\Dd' = \overline{\nu(\tau')} \cap (\Dd_2 \cup \Dd_3)$.   Then $X'$ is a topological 4--ball intersecting $X_1$ in two 3--balls in $\partial X_1$; i.e. $X'$ is a 1--handle.  It follows that $X_1'$ is obtained from $X_1$ by the attaching a 1--handle, so $X_1' \cong\natural^{k_1+1}(S^1\times B^3)$.  Similarly, $\Dd'$ is a band connecting disks $D_1$ and $D_2$ in $\Dd_1$.  Since these disks are trivial, we can assume without loss of generality that $D_1$ and $D_2$ have been isotoped to lie in $\partial X_1$, and since $\Dd'$ is boundary parallel inside $X'$, the disk $D' = D_1\cup\Dd'\cup D_2$ is boundary parallel in $X_1'$.  It follows that $\Dd_1' = \Dd_1\setminus (D_1\cup D_2)\cup D'$ is a trivial $(c_1-1)$--disk system.
	
	It remains to verify that the 3--dimensional components of the new construction are trivial tangles in handlebodies.  Observe that for $\{j,k\} = \{2,3\}$, the decomposition $(\pd X_j,\pd \Dd_j) = (H_{1j}',\tau_{1j}') \cup (H_{jk}',\tau_{jk}')$ is a 3--dimensional meridional stabilization of $(\pd X_j,\pd \Dd_j) = (H_{1j},\tau_{1j}) \cup (H_{jk},\tau_{jk})$.  Thus, $\tau_{ij}'$ is a trivial  $(b-1)$--strand tangle in the genus $g+1$ handlebody $H_{ij}'$, as desired.
\end{proof}

We can now prove Theorem~\ref{mainthm2}, which implies Corollary~\ref{maincor} as an immediate consequence.

\begin{proof}[Proof of Theorem~\ref{mainthm2}]
	Start with a generalized bridge trisection of $(X,\Kk)$.  If there is a spanning arc $\tau'$ of the type that is necessary and sufficient for a meridional stabilization, then we perform the stabilization.  Thus, we assume there are no such spanning arcs.  If $\Dd_i$ contains $c_i$ disks, then since there are no $\tau'$--type arcs in $\tau_{jk}$ for $i,j,k$ distinct, it follows that the $c_i$ disks belong to distinct connected components of $\Kk$.  Thus, $c_i=n$, and $\chi(\Kk) = c_1+c_2+c_3-b = 3n-b$, so that $b = 3n - \chi(\Kk)$.
\end{proof}

\section{Uniqueness of generalized bridge trisections}\label{sec:unique}

In general, the types of splittings discussed in this article are not unique up to isotopy, but a general principle is that two splittings for a fixed space become isotopic after some number of generic operations, such as the meridional stabilization operation defined above. Given a Heegaard splitting $\Sigma$ of $Y$, \emph{stabilization} can be defined as performing the connected sum of $Y$ and $S^3$ along the standard genus one splitting of $S^3$, and the Reidemeister-Singer Theorem asserts that any two Heegaard splittings for $Y$ become isotopic after some number of stabilizations~\cite{Reidemeister_Zur-dreidimensionalen_1933, Singer_Three-dimensional_1933}.  Similarly, \emph{stabilization} for a trisection $\Tt$ of a 4--manifold $X$ can be viewed as taking the connected sum of $\Tt$ and the standard genus three trisection of $S^4$, and Gay and Kirby proved that any pair of trisections for $X$ become isotopic after some number of trisections~\cite{Gay-Kirby_Trisecting_2016}.

For classical bridge splittings $(Y,K) = (H_1,\tau_1) \cup_{\Sigma} (H_2,\tau_2)$ there are two generic operations (in additional to meridional stabilization).  One, called \emph{stabilization}, is a stabilization of the Heegaard surface $\Sigma$ performed away from the knot $K$, increasing the genus of the bridge splitting.  The other operation increases the number of arcs in $\tau_i$:  Any bridge splitting of $(Y,K)$ obtained by connected summing with the standard $(0,2)$--splitting of the unknot in $S^3$ is said to be obtained by \emph{elementary perturbation}, and any splitting obtained by a finite number of elementary perturbations is called a \emph{perturbation}.  As with Heegaard splittings and trisections, it is known that if two bridge splittings of $(Y,K)$ have the same underlying Heegaard splitting, then there is a bridge splitting that is isotopic to perturbations of each of the original splittings, called a common perturbation~\cite{Hayashi_Stable_1998,Zupan_Bridge_2013}, and thus any two bridge splittings of $(Y,K)$ become isotopic after stabilizations and perturbations.  The purpose of this section is to define perturbations for generalized bridge trisections and lay out steps toward a proof of a corresponding uniqueness theorem in this setting.

Let $L$ be an $n$--component unlink in $Y = \#^k(S^1 \X S^2)$.  The \emph{standard} bridge splitting of $L$ is defined to be the connected sum of the the standard genus $k$ Heegaard splitting of $Y$ with the standard (classical) $n$--bridge splitting of $L$ (the connected sum of $n$ copies of the 1--bridge splitting of the unknot). The first ingredient we will need to define perturbation is the following proposition, which uses a result in~\cite{BacSch_2005_Distance-and-bridge} and follows from a proof identical to that of Proposition 2.3 in~\cite{Meier-Zupan_Bridge_2017}.

\begin{proposition}\label{trivbridge}
Every bridge splitting of an unlink $L$ in $\#^k(S^1 \X S^2)$ is isotopic to some number of perturbations and stabilizations performed on the standard bridge splitting.
\end{proposition}

Consider a bridge trisection $\Tt$ for a knotted surface $\Kk \subset X$, with components notated as above.  Proposition~\ref{trivbridge} implies the key fact that $\Kk$ admits shadow diagram $((\A,a),(\n,b),(\g,c))$ such that a pair of collections of arcs, say $a$ and $b$ for convenience, do not meet in their interiors, and in addition, the union $a \cup b$ cuts out a collection of embedded disks $\Dd_*$ from the central surface $\Sigma$.  Choose a single component $D_*$ of these disks together with an embedded arc $\delta_*$ in $D_*$ which connects an arc $a' \in a$ to an arc $b' \in b$.  Note that $\Dd_*$ is a trivialization of the disks $\Dd_1\subset X_1$ bounded by $\tau_{31} \cup \tau_{12}$ in $\partial X_1 = H_{31} \cup H_{12}$, so that we may consider $\delta_*$ and $D_*$ to be embedded in the surface $\Kk$.  In addition, there is an isotopy of $\Dd_*$ in $\pd X_1$ pushing the shadows $a \cup b$ onto arcs in $\tau_{31} \cup \tau_{12}$, making $D_*$ transverse to $\Sigma$ and carrying $\delta_*$ to an embedded arc in $\partial X_1$ that meets the central surface $\Sigma$ in one point.

Let $\Delta$ be a rectangular neighborhood of $\delta_*$ in $D_*$, and consider the isotopy of $\Kk$, supported in $\Delta$, which pushes $\delta_* \subset\Kk$ away from $X_1$ in the direction normal to $\pd X_1$.  Let $\Kk'$ be the resulting embedding, which is isotopic to $\Kk$.  The next lemma follows from the proof of Lemma~6.1 in~\cite{Meier-Zupan_Bridge_2017}.

\begin{lemma}
The embedding $\Kk'$ is in $(b+1)$--bridge position with respect to the trisection $X = X_1 \cup X_2 \cup X_3$, and if $c_i' = | \Kk' \cap X_i|$, then $c_1' = c_1+1$, $c_2' = c_2$, and $c_3' = c_3$.
\end{lemma}

We call the resulting bridge trisection an \emph{elementary perturbation} of $\Tt$, and if $\Tt'$ is the result of some number of elementary perturbations performed on $\Tt$, we call $\Tt'$ a \emph{perturbation} of $\Tt$.  Work in~\cite{Meier-Zupan_Bridge_2017} also makes clear how to perturb via a shadow diagram. View the rectangle $\Delta$ as being contained in $\Sigma$, and parameterize it as $\Delta = \delta_*\times I$.  Now, crush $\Delta$ to a single arc $c'=\ast\times I$ that meets $\delta_*$ transversely once.  Considering the arc $c'$ as a shadow arc for the third tangle, the result is a shadow diagram for the elementary perturbation of $\Tt$.  See Figure~\ref{fig:pert}.

\begin{figure}[h!]
	\centering
	\includegraphics[width=.8\textwidth]{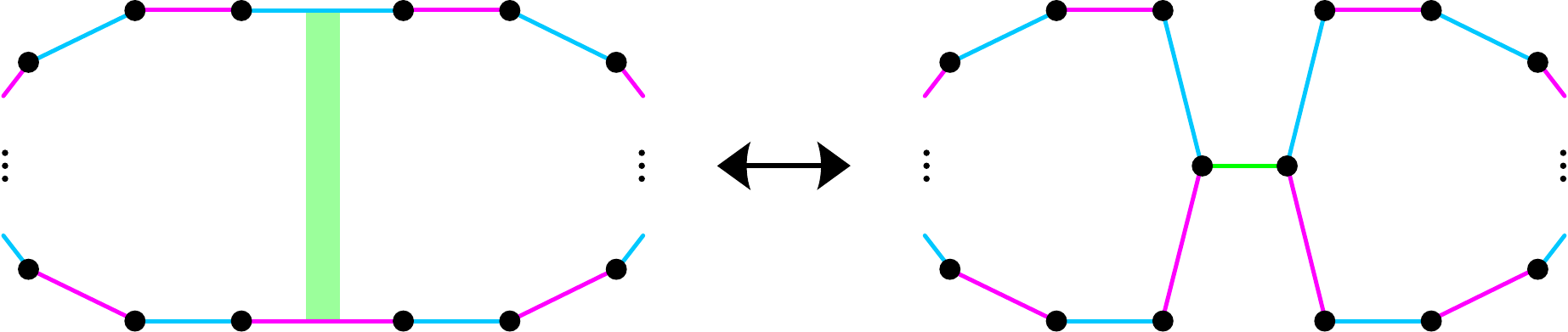}
	\caption{An illustration (at the level of the shadow diagram) of an elementary 1--perturbation of a generalized bridge splitting.}
	\label{fig:pert}
\end{figure}

In~\cite{Meier-Zupan_Bridge_2017}, the authors prove that any two bridge trisections for a knotted surface $(S^4,\Kk)$ are related by a sequence of perturbations and unperturbations.  In the setting of generalized bridge trisections, we have the following conjecture.

\begin{conjecture}\label{unique}
Any two generalized bridge trisections for $(X,\Kk)$ with the same underlying trisection for $X$ become isotopic after a finite sequence of perturbations and unperturbations.
\end{conjecture}

The proof of the analogous result for bridge trisections in~\cite{Meier-Zupan_Bridge_2017} requires a result of Swenton~\cite{Swenton_On-a-calculus_2001} and Kearton-Kurlin~\cite{Kearton-Kurlin_All-2-dimensional_2008} that states that every one-parameter family of Morse functions of the pair $h_t:(S^4,\Kk) \rightarrow \R$ such that $h_t:S^4 \rightarrow \R$ is the standard height function can be made suitably generic.  Unfortunately, a more general result does not yet exist for arbitrary pairs $(X,\Kk)$; however, we remark that Conjecture~\ref{unique} would follow from such a result together with an adaptation of the proof in~\cite{Meier-Zupan_Bridge_2017}.

\newpage

\bibliographystyle{amsalpha}
\bibliography{MasterBibliography_2017_08}

\end{document}